\documentclass{elsarticle}

\usepackage{latexsym,amssymb,amsmath,enumerate,bbm,graphicx,psfrag,color}
\usepackage[latin1]{inputenc}
\usepackage{url}

\definecolor{mygray}{rgb}{0.3333,0.3333,0.3333} 

\newcommand{\D}{\ensuremath{\mathbbm{D}}} 
\newcommand{\E}{\ensuremath{\mathbbm{E}}} 
\newcommand{\R}{\ensuremath{\mathbbm{R}}} 
\newcommand{\N}{\ensuremath{\mathbbm{N}}} 
\newcommand{\T}{\ensuremath{\mathbbm{T}}} 
\newcommand{\IP}{\ensuremath{\mathbbm{P}}} 
\newcommand{\Z}{\ensuremath{\mathbbm{Z}}} 
\newcommand{\cB}{\ensuremath{\mathcal{B}}}
\newcommand{\cC}{\ensuremath{\mathcal{C}}}
\newcommand{\cE}{\ensuremath{\mathcal{E}}} 
\newcommand{\cF}{\ensuremath{\mathcal{F}}} 
\newcommand{\cM}{\ensuremath{\mathcal{M}}} 
\newcommand{\cS}{\ensuremath{\mathcal{S}}}
\newcommand{\bB}{\ensuremath{\mathbf{B}}}

\newcommand{\dom}{\ensuremath{\mathrm{dom}}}

\newtheorem{theorem}{Theorem}[section]
\newtheorem{lemma}[theorem]{Lemma}

\newtheorem{definition}[theorem]{Definition}

\newtheorem{assumption}[theorem]{Assumption}
\newtheorem{proposition}[theorem]{Proposition}

\newproof{proof}{Proof}

\begin{document}
\begin{frontmatter}

\title{Weak convergence of Galerkin approximations of stochastic partial differential equations driven by additive L\'evy noise}

\author[ab]{Tobias St\"uwe\corref{cor1}\fnref{fn1}}
\ead{tobias.stuewe@math.uni-stuttgart.de}

\author[ab]{Andrea Barth\fnref{fn1}}
\ead{andrea.barth@math.uni-stuttgart.de}

\cortext[cor1]{Corresponding author}
\fntext[fn1]{The research leading to these results has received
funding from the German Research Foundation (DFG) as part of the Cluster of Excellence
in Simulation Technology (EXC 310/2) at the University of Stuttgart, as well as of the Landesstiftung Baden--W\"urttemberg and 
it is gratefully acknowledged.}

\address[ab]{SimTech, University of Stuttgart, Pfaffenwaldring 5a, 70569 Stuttgart, Germany}

\begin{abstract}
This work considers weak approximations of stochastic partial differential equations (SPDEs) driven by L\'evy noise. The SPDEs at hand are parabolic with additive noise processes. A weak-convergence rate for the corresponding Galerkin approximation is derived. The convergence result is derived by use of the Malliavin derivative rather then the common approach via the Kolmogorov backward equation.
\end{abstract}

\begin{keyword}Weak convergence \sep stochastic partial differential equation \sep L\'evy noise \sep Malliavin calculus \MSC{60H15 \sep 65C30 \sep 65M60 \sep 65M70}\end{keyword}
 

\end{frontmatter}

\pagestyle{myheadings}
\thispagestyle{plain}
\markboth{A. BARTH AND T. ST\"UWE}{WEAK CONVERGENCE OF SPDES WITH L\'EVY NOISE}


\section{Introduction}
\label{sec:intro}

In contrast to partial differential equations, the error analysis of approximations of solutions to stochastic (partial) differential equations (SPDEs) allows for two conceptually different approaches: weak and strong. Both of these kinds of error analysis for SPDEs have been actively researched during the last two decades. While the strong (or pathwise) error has been the subject of a vast array of publications, the weak error, which is computed in terms of moments of the solution process, has, to the date, garnered considerably less attention.

In this paper we consider weak-convergence rates of Galerkin approximations of solutions to the parabolic stochastic partial differential equation given, for $t\in(0,T]=:\T$, by

\begin{equation}\label{eq:SPDE}
\begin{split}
&dX(t)+AX(t)dt=f(t)dt+G(t)dL(t),\\
&X(0)=x_{0}\in H.
\end{split}
\end{equation}
By $H$ we denote a separable Hilbert space, $A$ is a linear operator on $H$, $f$ maps $\T$ into $H$ and $G$ is a mapping from $\T$ into the linear, bounded operators form some separable Hilbert space $U$ (not necessarily equal to $H$) into $H$. Further, $L$ denotes a L\'evy process defined on the complete probability space $(\Omega,\cF,\IP)$  and it takes values in $U$.

For solutions to SPDEs like Equation~\eqref{eq:SPDE}, the strong-error rate of Galerkin approximations has been considered, among others, in~\cite{B10, BaLa12_2, BL12, BaLa13, DaGa01, GrKl96, GyNu97, Ha03, KoLaLi10, KoLaSa10, KoLiLa11, MuRi07}. In these references, SPDEs driven by either Gaussian or L\'evy noises are treated. Publications on weak approximations and their error analysis are, among others,~\cite{AKL15_2, AKL15, BaLa12, BaLaSc13, De11, GeKoLa09, Ha03a, KoLa12, KLS15, K14}, where, to a great extent, SPDEs driven by Gaussian processes are considered. 

In this paper, we consider an equivalent approach as in~\cite{K14} and combine it with the recent results on Malliavin calculus for L\'evy driven SPDEs in~\cite{DMvN13}. In our main result, Theorem~\ref{thm:weak}, we show that the weak-convergence rate is twice the strong-convergence rate. This is akin to the findings in~\cite{KLS15}, where a similar equation is treated and a weak-convergence result for a Galerkin approximation via the backward Kolmogorov equation is derived. Our methodology, however, differs considerably and, with it, the regularity assumptions required on the functional of the solution.

The paper is organized as follows: In Section~\ref{sec:prelim} we provide the notation and the results on Malliavin calculus for infinite dimensional L\'evy processes required for the weak-convergence result. In the third section, we introduce the stochastic partial differential equation in question, as well as its approximation. We then proceed with the proof of the main result on weak convergence of this approximation. 

\section{Notation and preliminaries}
\label{sec:prelim}
Let $(U,\left\langle\cdot,\cdot\right\rangle_{U})$ and $(H,\left\langle\cdot,\cdot\right\rangle_{H})$ be separable Hilbert spaces and let $L(U;H)$ be the space of all linear bounded operators from $U$ into $H$ endowed with the usual supremum norm. If $U=H$, the abbreviation $L(U):=L(U;U)$ is used. An element $G\in L(U;H)$ is said to be a nuclear operator if there exists a sequence $(x_k, k\in\N)$ in $H$ and a sequence $(y_k, k\in\N)$ in $U$ such that 
\begin{equation*}
\sum_{k\in\N} \left\|x_k\right\|_{H} \left\|y_k\right\|_{U} < +\infty
\end{equation*} 
and $G$ has, for $\in U$ the representation
\begin{equation*}
Gz = \sum_{k\in\N} \left\langle z, y_k \right\rangle_{U} x_k.
\end{equation*}
The space of all nuclear operators from $U$ into $H$, endowed with the norm 
\begin{equation*}
\left\|G\right\|_{L_{N}(U;H)} := \inf\left\{\sum_{k\in\N} \left\|x_k\right\|_{H} \left\|y_k\right\|_{U}\;|\;Gz = \sum_{k=1}^{\infty} \left\langle z, y_k \right\rangle_{U} x_k\right\}
\end{equation*}
is a Banach space, and is denoted by $L_{N}(U;H)$. If $U=H$, we use the abbreviation $L_{N}(U)$. Furthermore, let $L_{N}^{+}(U)$ denote the space of all nonnegative, symmetric, nuclear operators on $U$, i.e.,
\begin{equation*}
L_{N}^{+}(U) := \left\{G\in L_{N}(U) | \left\langle Gy,y\right\rangle_{U} \geq 0,\;\left\langle Gy,z\right\rangle_{U} = \left\langle y,Gz\right\rangle_{U}\,\text{ for all } y,z\in U\right\}.
\end{equation*}
An operator $G\in L(U;H)$ is called a Hilbert-Schmidt operator if
\begin{equation*}
\left\|G\right\|_{L_{HS}(U;H)}^{2} := \sum_{k=1}^{\infty}\left\|Ge_k\right\|_{H}^{2} < +\infty
\end{equation*}
for any orthonormal basis $(e_k, k\in\N)$ of $U$. The space of all Hilbert-Schmidt operators $(L_{HS}(U;H),\left\|\cdot\right\|_{L_{HS}(U;H)})$ is a Hilbert space with inner product given by
\begin{equation*}
\left\langle G,\tilde{G}\right\rangle_{L_{HS}(U;H)} := \sum_{k=1}^{\infty}\left\langle Ge_k , \tilde{G}e_{k}\right\rangle_{H},
\end{equation*}
for $G,\;\tilde{G}\in L_{HS}(U;H)$ and any orthonormal basis $(e_k, k\in\N)$ of $U$. If $U=H$, the abbreviation $L_{HS}(U):=L_{HS}(U;U)$ is used.\\
Given a measure space $(S,\cS,\mu)$ and $r\in[1,+\infty)$, we denote by $L^r(S;H)$ the space of all $\cS\text{-}\mathcal{B}(H)$-measurable mappings $f:S\rightarrow H$ with finite norm 
\begin{equation*}
 \left\|f\right\|_{L^r(S;H)} := \left(\int_{S}\left\|f\right\|_{H}^{r}\;d\mu\right)^{\frac{1}{r}},
\end{equation*}
where $\mathcal{B}(H)$ denotes the Borel $\sigma$-algebra over $H$.

We consider stochastic processes on the time interval $\T:=[0,T]$, with $0<T<+\infty$, defined on a filtered probability space $(\Omega,\cF,(\cF_{t},t\in\T),\IP)$ satisfying the usual conditions. We denote by $\cM_{\T}^{2}(H)$ the space of all $H$-valued, c\`adl\`ag, square integrable martingales. The space $\cM_{\T}^{2}(H)$ equipped with the norm $\left\|\cdot\right\|_{\cM_{\T}^{2}(H)}$, which is defined by
\begin{equation*}
\left\|Y\right\|_{\cM_{\T}^{2}(H)} := \sup_{t\in\T}\left(\E\left[\left\|Y(t)\right\|_{H}^{2}\right]\right)^{\frac{1}{2}} = \left(\E\left[\left\|Y(T)\right\|_{H}^{2}\right]\right)^{\frac{1}{2}}
\end{equation*}
for $Y\in\cM_{\T}^{2}(H)$, is a Banach space.
\subsection{Stochastic integration with respect to compensated Poisson random measures}

Let 
$(S,\Sigma, \nu)$ be a $\sigma$-finite measure space. We introduce the notation $\bar{\Z}_{+} := \Z_{+}\cup\{+\infty\}$. We, further, work on the measure space $(S\times \T, \Sigma\otimes\cB(\T)) := (S_\T,\Sigma_\T)$ and denote by $\lambda$ the Lebesgue measure on $(\T,\cB(\T))$. 
\begin{definition}\label{def:PRM}
A Poisson random measure on $(S_\T,\Sigma_\T)$ with intensity measure $\mu:=\nu\otimes\lambda$ is a mapping $p:\Omega\times\Sigma_\T\rightarrow\bar{\mathbb{Z}}_{+}$ such that:
\begin{enumerate}
\item For all $\omega\in\Omega$, the mapping $B\mapsto p(\omega,B)$ is a measure,
\item For all $B\in\Sigma_\T$, the mapping $p(B): \omega \mapsto p(\omega,B)$ is a Poisson distributed random variable with parameter $\mu(B)$,
\item For any pairwise disjoint $B_1,...,B_M\in\Sigma_\T$, $M\in\N$, the random variables $p(B_1), ...,p(B_M)$ are independent.
\end{enumerate}
For $B\in\Sigma_\T$ with $\mu(B)<\infty$ we write
\begin{equation*}
q(B) := p(B) - \mu(B)
\end{equation*}
and call $q$ the compensated Poisson random measure associated to $p$.
\end{definition}

We assume that the underlying probability space is equipped with the filtration $\cF = (\cF_{t},t \in \T)$ generated by $q$, i.e.,
\begin{equation*} 
\cF_{t} := \sigma(q((r,s]\times A) | 0\leq r<s\leq t, A \in \Sigma, \nu(A) < \infty).
\end{equation*}
With slight abuse of notation, we write $q(t,A) := q((0,t]\times A)$, for $t \in \T$ and $A \in \Sigma$ with $\nu(A) < \infty$. Defined like this we have that for all $s,t \in \T$ with $s < t$, the increment $q(t,A) - q(s,A)$ is independent of $\cF_{s}$ and that $(q(t,A),t\in \T)$ is a square integrable $\cF$-martingale.

As it is common, we start to define the stochastic integral with respect to a compensated Poisson random measure, by considering elementary processes.
\begin{definition}
An $H$-valued stochastic process $\Phi : \Omega \times \T \times S \rightarrow H$ is said to be elementary if there exists some finite partition of $\T$, given by $0 = t_0 < ... < t_N = T$, for some $N \in \N$, and for every $n = 0, ..., N-1$ there exist pairwise disjoint sets $A_{1}^{n}, ..., A_{M_n}^{n} \in \Sigma$ of finite $\nu$-measure, such that
\begin{equation}\label{elem}
\Phi = \sum_{n=0}^{N-1}\sum_{m=1}^{M_n} \Phi_{m}^{n}\mathbf{1}_{(t_{n},t_{n+1}] \times A_{m}^n}
\end{equation}
where $\Phi_{m}^{n} \in L^{2}(\Omega;H)$ is $\cF_{t_n}$-measurable, for $m=1,...,M_n$, $n=0,...,N-1$. The class of all elementary processes is denoted by $\cE$.

For $\Phi \in \mathcal{E}$ we define the stochastic integral with respect to the compensated Poisson random measure $q$, for $t \in \T$, by
\begin{align*}
I(\Phi)(t) &:= \int_{(0,t]}\int_{S} \Phi(s,z) \;q(ds,dz)\\
&:= \sum_{n=0}^{N-1}\sum_{m=1}^{M_n} \Phi_{m}^{n}\bigl(q(t_{n+1}\wedge t,A_{m}^n) - q(t_{n}\wedge t,A_{m}^n)\bigr).
\end{align*}
\end{definition}
For every $\Phi \in \cE$, $I(\Phi) = (I(\Phi)(t),t\in\T)$ is a c\`adl\`ag, square integrable $H$-valued $\cF$-martingale, i.e., $I(\Phi) \in \cM_{\T}^{2}(H)$.

We endow the class of all elementary processes $\cE$ with the seminorm
\begin{equation*}
\|\Phi\|_{S_\T}^2 := \E\int_{0}^{T}\int_{S}\|\Phi(s,z)\|_{H}^2\;\nu(dz) ds.
\end{equation*}
To define a norm on $\cE$, we identify $\Phi$ with $\Psi$ if $\|\Phi - \Psi\|_{S_\T} = 0$. Then, the stochastic integral $I$ is an isometric mapping from $(\cE,\|\cdot\|_{S_\T})$ to $(\cM_{\T}^{2}(H),\|\cdot\|_{\cM_{\T}^{2}(H)})$, i.e., for $\Phi \in \cE$,
\begin{equation*}
\|I(\Phi)\|_{\cM_{\T}^{2}(H)} = \|\Phi\|_{S_\T}.
\end{equation*}
Let $\overline{\mathcal{E}}^{S_{\T}}$ be the completion of $(\cE,\|\cdot\|_{S_\T})$. It is clear that there is a unique isometric extension of $I$ to $\overline{\mathcal{E}}^{S_{\T}}$. This broader class of integrands is denoted by $\mathcal{N}_{q}^{2}(S_\T;H)$ and can be characterized by
\begin{equation*}
\mathcal{N}_{q}^{2}(S_\T;H) := L^{2}(\Omega\times\T\times S,\mathcal{P}_{\T}(S),\IP\otimes\lambda\otimes\nu;H),
\end{equation*}
where $\mathcal{P}_{\T}(S)$ denotes the $\sigma$-algebra of predictable sets in $\Omega\times\T\times S$, i.e.,
\begin{align*}
\mathcal{P}_{\T}(S) := \sigma\bigl(\{F_{s}\times (s,t]&\times A | 0\leq s<t \leq T, F_{s}\in\cF_{s}, A\in\Sigma\}\\
&\cup \{F_{0}\times\left\{0\right\}\times A | F_{0}\in\cF_{0}, A\in\Sigma\}\bigr).
\end{align*}
\subsection{The Malliavin derivative}
With the stochastic integral in hand, we shall outline the notion of the Malliavin derivative as introduced in \cite[Section 5]{DMvN13}. Given a Poisson random measure $p$ on $(S_\T,\Sigma_\T)$ with intensity measure $\mu = \nu\times\lambda$, it will be convenient to introduce the following notation. We write
for a tuple $\bB = (B_{1},...,B_{M})$, $B_{1},...,B_{M}\in\Sigma_\T$, $M\in\N$,
\begin{equation*}
p(\bB):=(p(B_1),...,p(B_M)),
\end{equation*}
and denote by $e_{m}$ the $m$-th unit vector in $\R^M$.

\begin{definition}
An $H$-valued random variable $F:\Omega\rightarrow H$ is called cylindrical if it has the form
\begin{equation}\label{cyl}
F = \sum_{i=1}^{n}f_{i}(p(B_1),...,p(B_M))h_{i},
\end{equation}
where $B_{m}\in\Sigma_\T$, with $\mu(B_M)<+\infty$ for $m=1,...,M$, $f_{i}:\Z_{+}^{M}\rightarrow\R$ and $h_i\in H$ for $i=1,...,n$, for some $M,n\in\N$. The collection of all $H$-valued cylindrical random variables is denoted by $\cC(\Omega;H)$.
\end{definition}

In the sequel, we always assume that the sets $B_{1},...,B_{M}$ used in the representation of $F\in\cC(\Omega;H)$ are pairwise disjoint. 
\begin{definition}
For a cylindrical random variable $F\in\cC(\Omega;H)$, the Malliavin derivative $DF\in L^{2}(\Omega\times S_\T;H)$ is defined by
\begin{equation*}
DF := \sum_{i=1}^{n}\sum_{m=1}^{M}\left(f_{i}(p(\mathbf{B})+e_{m})-f_{i}(p(\mathbf{B}))\right)\mathbf{1}_{B_m}h_{i}.
\end{equation*} 
\end{definition}
Note that the expression of the derivative does not depend on the particular representation of $F\in\cC(\Omega;H)$. The proof of the following proposition can be found in~\cite[Theorem 5.6]{DMvN13}.

\begin{proposition}
The operator $D:\cC(\Omega;H)\subset L^2(\Omega;H)\rightarrow L^{2}(\Omega\times S_\T;H)$ is closable.
\end{proposition}
%

By abuse of notation, we let $D$ stand for the closure of $D:\cC(\Omega;H)\subset L^2(\Omega;H)\rightarrow L^{2}(\Omega\times S_\T;H)$. We denote by $\D^{1,2}(\Omega;H)$ the domain of this closure which is a Banach space endowed with the norm
\begin{equation*}
\left\|F\right\|_{\D^{1,2}(\Omega;H)}:=\left(\left\|F\right\|_{L^2(\Omega;H)}^2 + \left\|DF\right\|_{L^{2}(\Omega\times S_\T;H)}^2\right)^{\frac{1}{2}},
\end{equation*} 
for $F\in\D^{1,2}(\Omega;H)$
\begin{proposition}
Let $\phi:H\rightarrow\tilde{H}$ be Lipschitz-continuous, where $(\tilde{H},\left\langle \cdot,\cdot\right\rangle_{\tilde{H}})$ is an arbitrary separable Hilbert space, and let $F\in\D^{1,2}(\Omega;H)$. Then $\phi(F)\in\mathbb{D}^{1,2}(\Omega;\tilde{H})$ with derivative
\begin{equation}\label{eq:chain}
D\phi(F)=\phi(F+DF)-\phi(F).
\end{equation}
\end{proposition}
\begin{proof}
First, we assume that $F\in\cC(\Omega;H)$. 
Consider the sequence $(\phi_\ell, \ell\in\N)$ of functions $\phi_{\ell}:H\rightarrow\tilde{H}$ defined by
\begin{equation*}
\phi_{\ell}(h):=\sum_{k=1}^{\ell}\left\langle\phi(h),\tilde{h}_k\right\rangle_{\tilde{H}}\tilde{h}_{k},\quad h\in H,
\end{equation*}
where $(\tilde{h}_{k}, k\in\N)$ denotes an arbitrary orthonormal basis of $\tilde{H}$. Now, for every $\ell\in\N$ we get
\begin{equation*}
\phi_{\ell}(F)=\sum_{k=1}^{\ell}g_{k}(p(B_1),...,p(B_M))\tilde{h}_{k},
\end{equation*}
where the functions $g_{k}:\Z_{+}^{M}\rightarrow\R$, $k\in\N$, are given by
\begin{equation*}
g_{k}(\mathbf{m}):=\langle\phi(\sum_{i=1}^{n}f_{i}(\mathbf{m})h_i),\tilde{h}_{k}\rangle_{\tilde{H}},\quad \mathbf{m}=(m_{1},...,m_{M})\in\Z_{+}^M.
\end{equation*}
From this follows that, for all $\ell\in\N$, $\phi_{\ell}(F)\in\cC(\Omega;\tilde{H})$ with derivative
\begin{align*}
D\phi_{\ell}(F)&=\sum_{k=1}^{\ell}\sum_{m=1}^{M}\left(g_{k}(p(\bB)+e_m)-g_{k}(p(\bB))\right)\mathbf{1}_{B_m}\tilde{h}_{k}\\
&=\sum_{m=1}^{M}\mathbf{1}_{B_m}\left(\phi_{\ell}(\sum_{i=1}^{n}f_{i}(p(\mathbf{B})+e_m)h_i) - \phi_{\ell}(\sum_{i=1}^{n}f_{i}(p(\mathbf{B}))h_i)\right)\\
&=\phi_{\ell}(F+DF)-\phi_{\ell}(F).
\end{align*}
If we prove that, for $\ell\rightarrow\infty$,
\begin{equation}\label{eq:convergence}
\begin{split}
&\phi_{\ell}(F)\rightarrow\phi(F)\;\text{in}\;L^{2}(\Omega;\tilde{H})\;\text{and}\\
&D\phi_{\ell}(F)\rightarrow\phi(F+DF)-\phi(F)\;\text{in}\;L^{2}(\Omega\times S_\T;\tilde{H}),
\end{split}
\end{equation}
Equation~\eqref{eq:chain} follows for $\phi$, by the closedness of $D$. By the definition of $(\phi_\ell,\ell\in\N)$ we have, for every $h\in H$, $\phi_{\ell}(h)\rightarrow\phi(h)$ as $\ell\rightarrow\infty$. This clearly forces convergence a.e. in Equation~\eqref{eq:convergence}. If we can find dominating functions the desired result follows from Lebesgue's dominated convergence theorem. Indeed, we have
\begin{equation*}
\left\|\phi_{\ell}(F)-\phi(F)\right\|_{\tilde{H}}\leq\left\|\phi(F)\right\|_{\tilde{H}}\leq C\left(1+\left\|F\right\|_{H}\right),
\end{equation*}
where we used the Lipschitz-property of $\phi$ and that, for all $x\in H,\;\ell\in\N$,
\begin{equation*}
\left\|\phi_{\ell}(x)-\phi(x)\right\|_{\tilde{H}}\leq\left\|\phi(x)\right\|_{\tilde{H}}.
\end{equation*}
Further, we have
\begin{align*}
&\left\|D\phi_{\ell}(F)-(\phi(F+DF)-\phi(F))\right\|_{\tilde{H}}\\
&\qquad\leq\left\|\phi_{\ell}(F+DF)-\phi_{\ell}(F)\right\|_{\tilde{H}}+\left\|\phi(F+DF)-\phi(F)\right\|_{\tilde{H}}\\
&\qquad\leq 2\left\|\phi(F+DF)-\phi(F)\right\|_{\tilde{H}}\\
&\qquad\leq C\left\|DF\right\|_{H},
\end{align*}
where we used that, for all $x,\;y\in H $ and $\ell\in\N$,
\begin{equation*}
\left\|\phi_{\ell}(x)-\phi_{\ell}(y)\right\|_{\tilde{H}}\leq\left\|\phi(x)-\phi(y)\right\|_{\tilde{H}}.
\end{equation*}

For $F\in\mathbb{D}^{1,2}(\Omega;H)$ arbitrary, we take a sequence $(F_{k},k\in\N)\subset\mathcal{C}(\Omega;H)$ such that
\begin{equation*}
F_k\rightarrow F\;\text{in}\;\mathbb{D}^{1,2}(\Omega;H)\;\text{as}\;k\rightarrow\infty.
\end{equation*}
Then, by the Lipschitz-continuity of $\phi$, it holds that
\begin{equation*}
\left\|\phi(F_k)-\phi(F)\right\|_{L^{2}(\Omega;\tilde{H})}\leq C\left\|F_{k}-F\right\|_{L^{2}(\Omega;H)},
\end{equation*}  
where the right hand side tends to zero for $k\rightarrow\infty$. By the closedness of $D$, the proof is completed by showing that 
\begin{equation*}\label{conv}
D\phi(F_k)\rightarrow\phi(F+DF)-\phi(F)\;\text{in}\;L^{2}(\Omega\times S_\T;\tilde{H})\;\text{as}\;k\rightarrow\infty.
\end{equation*}
For this purpose, we assume, by possibly considering a suitable subsequence of $(F_k,k\in\N)$, that $F_{k}\rightarrow F$ $\IP$-a.e. and $DF_{k}\rightarrow DF$ $\IP\otimes\mu$-a.e. as $k\rightarrow\infty$. Since $\phi$ is continuous, we obtain convergence $\IP\otimes\mu$-a.e..
By the above and the Lipschitz-continuity of $\phi$, we have
\begin{align*}
\|D\phi(F_k) &- (\phi(F+DF)-\phi(F))\|_{\tilde{H}}\\
&\leq \|\phi(F_{k}+DF_{k})-\phi(F_k)\|_{\tilde{H}} + \|\phi(F+DF)-\phi(F) \|_{\tilde{H}}\\
&\leq C\big(\|DF_k\|_{H} + \|DF\|_{H}\big),
\end{align*}
where the right-hand side converges in $L^{2}(\Omega\times S_\T)$. The result follows by the application of a generalized version of Lebesgue's dominated convergence theorem.    
\end{proof}

Thanks to~\cite[Lemma 5.7]{DMvN13}, the operator $D$ is densely defined. Therefore the following definition makes sense. 
\begin{definition}
The divergence operator
\begin{equation*}
\delta:\dom(\delta)\subset L^{2}(\Omega\times S_\T;H)\rightarrow L^2(\Omega;H)
\end{equation*}
is defined to be the adjoint of
\begin{equation*}
D:\D^{1,2}(\Omega;H)\subset L^2(\Omega;H)\rightarrow L^{2}(\Omega\times S_\T;H).
\end{equation*}
\end{definition}
From the definition, it is clear that, for all $F\in\D^{1,2}(\Omega;H)$ and $\Phi\in \dom(\delta)$,
\begin{equation}\label{adj}
\E\int_{S_\T}\left\langle DF,\Phi\right\rangle\;d\mu = \E\left[\left\langle F,\delta(\Phi)\right\rangle\right].
\end{equation}
\begin{lemma}
For every $\phi\in L^{2}(S_\T;H)$ it holds that $\delta(\phi)\in\D^{1,2}(\Omega;H)$ and $D(\delta(\phi)) = \phi$.
\end{lemma}
\begin{proof}
Suppose first that $\phi$ is a simple function
\begin{equation*}
\phi = \sum_{i=1}^{M}\mathbf{1}_{B_{i}}h_{i},
\end{equation*}
where $B_{i}\in\Sigma_{\mu}$ are pairwise disjoint sets and $h_{i}\in H$. By~\cite[Lemma 5.9]{DMvN13}, we have $\phi\in \dom(\delta)$ and
\begin{equation*}
\delta(\phi) = \sum_{i=1}^{M}q(B_i)h_{i} = \sum_{i=1}^{M}f_{i}(p(\mathbf{B}))h_{i},
\end{equation*}
where 
\begin{equation*}
f_{i}(\mathbf{n}) := n_{i}-\mu(B_i),\quad \mathbf{n}:=(n_{1},...,n_{M})\in\mathbb{Z}_{+}^{M}. 
\end{equation*}
Thus, we get $\delta(\phi)\in\mathcal{C}(\Omega;H)$ and
\begin{align*}
D\delta(\phi) &= \sum_{i=1}^{M}\sum_{m=1}^{M}\left(f_{i}(p(\mathbf{B})+e_{m})-f_{i}(p(\mathbf{B}))\right)\mathbf{1}_{B_m}h_{i}\\
&= \sum_{i=1}^{M}\left(f_{i}(p(\mathbf{B})+e_{i})-f_{i}(p(\mathbf{B}))\right)\mathbf{1}_{B_i}h_{i} = \sum_{i=1}^{M}\mathbf{1}_{B_{i}}h_{i} = \phi.
\end{align*}

Now, let $\phi\in L^{2}(S_\T;H)$ be arbitrary. We may choose a sequence $\left(\phi_\ell,\ell\in\N\right)$ of simple functions converging to $\phi$ in $L^{2}(S_\T;H)$. Since $\left\|\delta(\phi_\ell)\right\|_{L^{2}(\Omega;H)} = \left\|\phi_\ell\right\|_{L^{2}(S_\T;H)}$, for $\ell\in\N$, the sequence $\left(\delta(\phi_\ell),\ell\in\N\right)$ converges in $L^{2}(\Omega;H)$. By the closedness of $\delta$, we see that $\phi\in \dom(\delta)$ and $\delta(\phi_{\ell}) \rightarrow \delta(\phi)$ as $\ell\rightarrow\infty$. Furthermore, we have
\begin{equation*}
D\delta(\phi_\ell) = \phi_{\ell} \rightarrow \phi\; \text{in}\; L^{2}(\Omega\times S_\T;H)
\end{equation*}
The assertion follows by the fact that $D$ is closed. 
\end{proof}

We return to the situation where $p$ is a Poisson random measure on $(S_\T,\Sigma_\T)$
with intensity measure $\mu=\lambda\otimes\nu$. As before we consider the filtration generated by the corresponding compensated measure $q$.  
\begin{proposition}\label{prop:skorohod}
Let $\Phi\in\mathcal{N}_{q}^{2}(S_\T;H)$ be a predictable stochastic process. Then $\Phi\in \dom(\delta)$ and
\begin{equation}\label{eq:skorohod}
\delta(\Phi) = I(\Phi)(T).
\end{equation} 
\end{proposition}
\begin{proof}
Suppose first that $\Phi\in\cE$ is an elementary process as given in Equation~\eqref{elem} with $\Phi_{m}^{n}\in\cC(\Omega;H)$, which are $\cF_{t_n}$-measurable. As $\delta$ is linear we have, by~\cite[Lemma 5.9]{DMvN13},
\begin{equation*}
\delta(\Phi) = \sum_{n=0}^{N-1}\sum_{m=1}^{M_n}q((t_{n},t_{n+1}]\times A_{m}^{n})\Phi_{m}^{n}=I(\Phi)(T).
\end{equation*} 
Since $\cC(\Omega;H)$ is dense in $L^2(\Omega;H)$, the closedness of $\delta$ shows that $\cE$ is contained in $\dom(\delta)$ and Equation~\eqref{eq:skorohod} holds. Finally, since $\cE$ is dense in $\mathcal{N}_{q}^{2}(S_\T;H)$, applying the closedness of $\delta$ again, the assertion follows.
\end{proof}

With the Malliavin derivative in hand we are able to prove the main result on weak convergence.
\section{Weak error estimate for the parabolic SPDEs}
Assume $L=(L(t),t\in\T)$ is a L\'evy process in a real separable Hilbert space $(U,\left\langle\cdot , \cdot\right\rangle_{U})$ defined on a filtered probability space $(\Omega,\cF,(\cF_{t},t\in\T),\IP)$ satisfying the usual conditions. We assume that for $s, t\in\T$ with $s<t$ the increment $L(t)-L(s)$ is independent of $\cF_{s}$. We assume that $L$ is square integrable and of mean zero. It follows that $L$ is a martingale with respect to $(\cF_{t},t\in\T)$. It is well-known that $L$ is square integrable if and only if its L\'evy measure $\nu$ satisfies
\begin{equation}\label{eq:levymeas}
\int_{U}\left\|u\right\|_{U}^{2}\;\nu(du) < +\infty.
\end{equation}
Moreover, we assume that the Gaussian part of $L$ vanishes. We always consider a c\`adl\`ag modification of $L$ and define the jump process of $L$ by $\Delta L(t):=L(t)-L(t-)$, for $t\in\T$. Let $Q\in L_{N}^{+}(U)$ be the covariance operator of $L$, which is determined by the L\'evy measure $\nu$ via
\begin{equation*}
\left\langle Qx,z\right\rangle_{U} = \int_{U}\left\langle x,u\right\rangle_{U}\left\langle z,u\right\rangle_{U}\;\nu(du), \qquad x,z\in U.
\end{equation*}
We introduce the space $U_{0}:=Q^{\frac{1}{2}}(U)$ which endowed with the inner product
\begin{equation*}
\left\langle x,z\right\rangle_{U_0} := \left\langle Q^{-\frac{1}{2}}x,Q^{-\frac{1}{2}}z\right\rangle_{U},\qquad x, z \in U_{0},
\end{equation*}
becomes a separable Hilbert space, called the reproducing kernel Hilbert space of $L$. Here, $Q^{-\frac{1}{2}}$ denotes the pseudo-inverse of $Q^{\frac{1}{2}}$. Since $Q$ is nuclear, $Q^{\frac{1}{2}}$ is Hilbert-Schmidt.
Consequently, the embedding $U_{0}\hookrightarrow U$ is Hilbert-Schmidt, i.e., for arbitrary orthonormal basis $\left(e_{i},i\in\N\right)$ of $U_0$ one has
\begin{equation*}
\sum_{i\in\N}\left\|e_i\right\|_{U}^{2} < +\infty.
\end{equation*}

Setting
\begin{equation*}
p:=\sum_{0<s\leq T}\mathbf{1}_{\left\{\Delta L(s)\neq 0\right\}}\delta_{(s,\Delta L(s))}
\end{equation*}
defines a Poisson random measure on $(\T\times U,\cB(\T)\otimes\cB(U))$ with intensity measure $\lambda\otimes\nu$. The associated compensated measure is denoted by
\begin{equation*}
q:=p-\lambda\otimes\nu.
\end{equation*}

In order to make the results of the previous section applicable, especially Proposition~\ref{prop:skorohod}, we assume that $\cF$ and the filtration $(\cF_{t},t\in\T)$ are generated by $q$.

Combining~\cite[Lemma A.2]{KLS15}, Proposition~\ref{prop:skorohod} with $(S,\Sigma)=(U,\cB(U))$ and Equation~\eqref{adj} we obtain
\begin{proposition}\label{prop:4.1}
If $\Phi\in L^{2}(\Omega\times\T,\mathcal{P}_{\T};L_{HS}(U_{0};H))$ and $F\in\mathbb{D}^{1,2}(\Omega;H)$, then
\begin{equation*}
\E\left[\left\langle F,\int_{0}^{T}\Phi(s)\;dL(s)\right\rangle_{H}\right] = \E\int_{0}^{T}\int_{U}\left\langle [DF](s,u),\Phi(s)u\right\rangle_{H}\;\nu(du)ds.
\end{equation*} 
\end{proposition}

Our main objective is to prove weak convergence of the mild solution $X:\Omega\times\T\rightarrow H$ to the stochastic evolution equation
\begin{equation}\label{eq:dX}
\begin{split}
&dX(t)+AX(t)dt=f(t)dt+G(t)dL(t)\\
&X(0)=x_{0}\in H
\end{split}
\end{equation}
We make the assumption:
\begin{assumption}
The data of the stochastic evolution equation~\eqref{eq:dX} satisfy:
\begin{enumerate}
\item The linear operator $A:\mathcal{D}(A)\subset H\rightarrow H$ is densely defined, self-adjoint, positive-definite and has a compact inverse.
\item The functions $f:\T\rightarrow H$ and $G:\T\rightarrow L(U;H)$ are measurable and bounded.
\end{enumerate}
\end{assumption}
Under these conditions, $-A$ is the generator of an analytic semigroup of contractions $(S(t), t\in\T)$. The mild solution is then given by the variation of constants formula
\begin{equation}\label{eq:mild}
X(t)=S(t)x_{0}+\int_{0}^{t}S(t-s)f(s)\;ds+\int_{0}^{t}S(t-s)G(s)\;dL(s).
\end{equation}

Furthermore, by the assumption there exists an nondecreasing sequence $(\lambda_{k}, k\in\N)$ of positive real numbers, which tends to $\infty$, and an orthonormal basis $(e_{k}, k\in\N)$ of $H$ such that $Ae_k = \lambda_{k}e_{k}$. This enables us to define fractional powers of the operator $A$ in the following way. 
  
For $s\geq 0$, $A^{\frac{s}{2}}:\mathcal{D}(A^{\frac{s}{2}})\subset H\rightarrow H$ is given by
\begin{equation*}
A^{\frac{s}{2}}x := \sum_{k=1}^{\infty}\lambda_{k}^{\frac{s}{2}}\left\langle x,e_{k}\right\rangle e_{k},
\end{equation*}  
for all $x\in\mathcal{D}(A^{\frac{s}{2}})$, where
\begin{equation*}
\mathcal{D}(A^{\frac{s}{2}}):=\Big\{x\in H | \left\|x\right\|_{s}^{2} := \sum_{k=1}^{\infty}\lambda_{k}^{s}\left\langle x, e_{k}\right\rangle^{2} < \infty\Big\}.
\end{equation*}
Then $\dot{H}^{s}:=\mathcal{D}(A^{\frac{s}{2}})$ endowed with the norm $\left\|\cdot\right\|_{s}$ becomes a Hilbert space. We may alternatively express the norm $\left\|\cdot\right\|_{s}$ as
\begin{equation*}
\left\|x\right\|_{s} = \left\|A^{\frac{s}{2}}x\right\|_{H},\qquad \text{for all}\;x\in\dot{H}^{s}.
\end{equation*} 

Let $(V_{h},h\in (0,1])$ be a family of finite dimensional subspaces of $\dot{H}^{1}$. Unless otherwise stated, we endow $V_{h}$ with the norm in $H$. By $P_{h}:H\rightarrow V_{h}$ and $R_{h}:\dot{H}^{1}\rightarrow V_{h}$ we denote the orthogonal projections with respect to the inner products in $H$ and $\dot{H}^{1}$, respectively. We assume that the Ritz projection $R_{h}$ satisfies the estimate
\begin{equation}\label{eq:Ritz}
\left\|R_{h}x-x\right\|_{H}\leq Ch^{\beta}\left\|x\right\|_{\beta}, \qquad x\in \dot{H}^{\beta}, \beta\in\{1,2\}, h\in (0,1].
\end{equation} 
Discrete versions $A_{h}:V_{h}\rightarrow V_{h}$ of the operator $A$ are then defined in the following way: For $x\in V_h$ we define $A_{h}x$ to be the unique element in $V_h$ for which
\begin{equation*}
\left\langle x,y\right\rangle_{\dot{H}^{1}} = \left\langle A_{h}x,y\right\rangle \text{ for all}\; y\in V_{h}.
\end{equation*}
Obviously, $A_h$ is self-adjoint and positive definite on $V_h$. Hence, $-A_{h}$ is the generator of an analytic semigroup of contractions on $V_h$, which is denoted by $S_{h}(t):=e^{-tA_{h}}$, for $t\in\T$. In what follows, we use the abbreviation 
\begin{equation}\label{eq:Fh}
F_{h}(t):=S_{h}(t)P_{h}-S(t), \quad t\geq 0.
\end{equation}
Given such a family of finite element spaces $V_{h}\subset \dot{H}^{1}$, we define, for $h\in (0,1]$, an approximation $(X_{h}(t), t\in\T)$ of the solution $(X(t),t\in\T)$ to be the mild solution to
\begin{equation}\label{eq:dXh}
\begin{split}
&dX_{h}(t)+A_{h}X_{h}(t)dt=P_{h}f(t)dt+P_{h}G(t)dL(t)\\
&X_{h}(0)=P_{h}x_{0}\in V_{h}
\end{split}
\end{equation}
Therefore $X_{h}:\Omega\times\T\rightarrow V_{h}$, $h\in (0,1]$, is given by
\begin{equation}\label{eq:mildXh}
X_{h}(t)=S_{h}(t)P_{h}x_{0}+\int_{0}^{t}S_{h}(t-s)P_{h}f(s)\;ds+\int_{0}^{t}S_{h}(t-s)P_{h}G(s)\;dL(s)
\end{equation}

The following deterministic estimate will be used in the proof of our weak error result stated in Theorem~\ref{thm:weak} below. For a proof, the reader is referred to~\cite[Theorem 3.2]{T06}. There the result is formulated under the assumption that $-A$ is the Laplace operator with homogeneous Dirichlet boundary conditions, the proof, however, can be extended to the more general setting we work in.
\begin{lemma}\label{lem:det_convergence}
Let Equation~\eqref{eq:Ritz} hold. Then there exists a constant $C>0$ such that for any $h\in(0,1]$ and $t>0$ 
\begin{equation*}
\left\|F_{h}(t)\right\|_{L(H)}\leq Ch^{2}t^{-1}.
\end{equation*}
\end{lemma}
With this result in hand we prove our main result on weak convergence.
\begin{theorem}\label{thm:weak}
Given a continuously Fr\'echet-differentiable mapping $\phi : H \rightarrow \mathbb{R}$ with Lipschitz continuous derivative, there exists a constant $C(T)>0$, independent of $h$, such that
\begin{equation*}\label{weakerr}
|\E\left[\phi(X_{h}(T))-\phi(X(T))\right]|\leq C(T)(1+|\ln(h)|)h^{2}.
\end{equation*}
\end{theorem}
\begin{proof}
The mean value theorem yields
\begin{align*}
&|\E\left[\phi(X_{h}(T)) - \phi(X(T))\right]|\\ 
&= \left|\E\left\langle \int_{0}^{1} \phi'\left(\sigma X_{h}(T)+(1-\sigma)X(T)\right) \; d\sigma, X_{h}(T) - X(T)\right\rangle_{H}\right|\\
&\leq \int_0^{1} \left|\E\left\langle \phi'\left(\sigma X_{h}(T)+(1-\sigma)X(T)\right), F_{h}(T)x_{0}\right\rangle_{H}\right| \; d\sigma\\
&\quad+ \int_0^{1} \left|\E\left\langle \phi'\left(\sigma X_{h}(T)+(1-\sigma)X(T)\right), \int_{0}^{T} F_{h}(T-s)f(s)\;ds\right\rangle_{H}\right| \; d\sigma\\
&\quad + \int_0^{1} \left|\E\left\langle \phi'\left(\sigma X_{h}(T)+(1-\sigma)X(T)\right), \int_{0}^{T} F_{h}(T-s)G(s)\;dL(s)\right\rangle_{H}\right| \; d\sigma\\
&=: I_{h}^{1}(T)+I_{h}^{2}(T)+I_{h}^{3}(T).
\end{align*}
To the first term, we apply the Cauchy-Schwarz inequality, the fact that $\phi'$ grows at most linearly and Lemma~\ref{lem:det_convergence}, to get
\begin{align*}
I_{h}^{1}(T) &\leq \int_{0}^{1} \left\|\phi'(\sigma X_{h}(T) + (1 - \sigma)X(T))\right\|_{L^{2}(\Omega;H)} \;d\sigma \left\|F_{h}(T)x_{0}\right\|_{H}\\
&\leq C \left(1 + \left\|X_{h}(T)\right\|_{L^{2}(\Omega;H)} + \left\|X(T)\right\|_{L^{2}(\Omega;H)}\right) \left\|F_{h}(T)\right\|_{L(H)}\left\|x_0\right\|_{H}\\
&\leq C T^{-1} h^{2}.
\end{align*}
Similarly, the second term is bounded by
\small
\begin{align*}
I_{h}^{2}(T) &\leq \int_{0}^{1} \left\|\phi'(\sigma X_{h}(T) + (1 - \sigma)X(T))\right\|_{L^{2}(\Omega;H)} \;d\sigma \left\|\int_{0}^{T} F_{h}(T-s)f(s)\;ds\right\|_{H}\\
&\leq C \left(1 + \left\|X_{h}(T)\right\|_{L^{2}(\Omega;H)} + \left\|X(T)\right\|_{L^{2}(\Omega;H)}\right) \int_{0}^{T} \left\|F_{h}(T - s)\right\|_{L(H)}\left\|f(s)\right\|_{H} \;ds\\
&\leq C\left(\int_{0}^{T-h^2} \left\|F_{h}(T - s)\right\|_{L(H)} \;ds+\int_{T-h^{2}}^{T} \left\|F_{h}(T - s)\right\|_{L(H)} \;ds\right)\\
&\leq C\left(h^2\int_{0}^{T-h^2}(T-s)^{-1}\;ds+h^{2}\right)\\
&\leq Ch^2\left(\ln(T)-2\ln(h) + 1\right)\\
&\leq C(T) (1 + |\ln(h)|) h^{2},
\end{align*}
\normalsize
where we used the boundedness of $f$ in the third step and Lemma~\ref{lem:det_convergence} to estimate the first summand of the third line.

It remains to estimate $I_{h}^{3}(T)$. Applying Proposition~\ref{prop:4.1}, the chain rule and the Lipschitz continuity of $\phi'$, Lemma~\ref{lem:det_convergence} and the boundedness of $G$ yields
\small
\begin{align*}
&I_{h}^{3}(T)\\
&=\int_{0}^{1}\Big|\E\int_{0}^{T}\int_{U}\left\langle D\phi'(\sigma X_{h}(T) + (1 - \sigma)X(T)),F_{h}(T-s)G(s)u\right\rangle_{H}\nu(du)ds\Big|\,d\sigma\\
&\leq C \int_{0}^{1} \E\Big[\int_{0}^{T}\int_{U} \left\|\sigma[DX_{h}(T)](s,u) + (1 - \sigma)[DX(T)](s,u)\right\|_{H} \\
&\qquad\qquad\qquad\qquad\qquad\qquad\qquad\qquad\qquad\qquad\qquad\left\|F_{h}(T - s)G(s)u\right\|_{H}\nu(du) ds\Big]d\sigma\\
&\leq C \E\Big[\int_{0}^{T}\int_{U} \left(\left\|[DX_{h}(T)](s,u)\right\|_{H} + \left\|[DX(T)](s,u)\right\|_{H}\right) \left\|F_{h}(T - s)G(s)u\right\|_{H}\nu(du) ds\Big]\\
&= C \int_{0}^{T}\int_{U} \left(\left\|S_{h}(T - s)P_{h}G(s)u\right\|_{H} + \left\|S(T - s)G(s)u\right\|_{H}\right) \left\|F_{h}(T - s)G(s)u\right\|_{H}\nu(du) ds\\
&\leq C\int_{U} \left\|u\right\|_{U}^{2}\;\nu(du)\int_{0}^{T} \left\|F_{h}(T - s)\right\|_{L(H)}\;ds\\
&\leq C(T) (1 + |\ln(h)|) h^{2},
\end{align*}
\normalsize
where the last estimate has already been used in the bound of $I_{h}^{2}(T)$. Summing the estimates proves then the assertion.
\end{proof}

\section*{References}

\bibliographystyle{model1b-num-names}
\bibliography{literature}

\end{document}